\documentclass[12pt, reqno]{amsart}
\usepackage{latexsym}

\newcommand{\haf}{{\frac{1}{2}}}
\newcommand{\holo}{\mathcal{O}}

\newcommand{\C}{{\mathbb C}}

\newcommand{\newsection}[1]
{\subsection{#1}\setcounter{theorem}{0} \setcounter{equation}{0}
\par\noindent}

\newtheorem{theorem}{Theorem}

\newtheorem{lemma}[theorem]{Lemma}

\newtheorem{proposition}[theorem]{Proposition}
\newtheorem{deff}[theorem]{Definition}

\newcommand{\eprop}{\end{proposition}}

\newcommand{\Sum}{\displaystyle\sum}
\newcommand{\Prod}{\displaystyle\prod}

\begin{document}
\thispagestyle{empty}

\noindent {\large {\bf The order of the decay of the hole probability for Gaussian random SU($m+1$)\\ polynomials.\hfill\\
by Scott Zrebiec }}\medskip

\bigskip\bigskip
{\bf Abstract} \par We show that for Gaussian random SU($m+1$)
polynomials of a large degree $N$ the probability that there are
no zeros in the disk of radius $r$ is less than $e^{-c_{1,r}
N^{m+1}}$, and is also greater than $e^{-c_{2,r} N^{m+1}}$.
Enroute to this result, we also derive a more general result:
probability estimates for the event where the volume of the zero
set of a random polynomial of high degree deviates significantly
from its mean.
\newsection{Introduction and notation}\par
A hole refers to the event where a particular Gaussian random
holomorphic function has no zeros in a given domain where many are
expected. The order of the decay of the hole probability has been
computed in several cases including for ``flat" complex Gaussian
random holomorphic functions on $\C^1$, \cite{SodinTsirelson05},
using a method which shall be used here. This work was
subsequently refined to cover other large deviations in the
distribution of the zeros sets, \cite{Krishnapur}, and generalized
to $\C^m$, \cite{Zrebiec06}. Other results compute the hole
probability for a class of complex Gaussian random holomorphic
functions on the unit disk, \cite{PeresVirag04}, and provide a
weak general estimate for any one variable complex Gaussian random
holomorphic functions, \cite{Sodin00}. Additionally significant
hole probability results have been discovered for real Gaussian
random polynomials, (\cite{DemboPoonenShaoZeitouni02},
\cite{LiShao02}).
\par
Various properties of the zeros of random SU($m+1 $) polynomials
have been studied, in particular the zero point correlation
functions have been computed. This is of particular interest in
the physics literature as the zeros describe a random spin state
for the Majorana representation (modulo phase), \cite{Hannay96}.
Further this choice is intuitively pleasing as the zeros are
uniformly distributed on $\C P^m$ (according to the Fubini-Study
metric), or alternatively the average distribution of zeros is
invariant under the SU($m+1$) action on $\C P^m$. These random
SU($m+1$) polynomials can be written as:
$$\psi_{\alpha,N}(z) = \sum_{|j|=0}^N \alpha_j \sqrt{ N
\choose{j}} z^j=\sum_{\sum j_k \leq N} \alpha_j \sqrt{ N
\choose{j_1, \ldots, j_m}} z_1^{j_1} z_2^{j_2}\ldots z_m^{j_m} ,
\eqno{(1)}$$ using standard multi index notation, and where
$\forall j, \ \alpha_j,$ are independent identically distributed
standard complex Gaussian random variables (mean 0 and
variance 1). 
\par For these Gaussian random SU($m+1$) polynomials we will be
computing the hole probability in a manner based on that used by
Sodin and Tsirelson to solve the similar problem for flat random
holomorphic functions on $\C^1$, \cite{SodinTsirelson05}. In
particular, we shall be estimating the unintegrated counting
function for a random SU($m+1$) variable polynomial, which is
defined as
$$n_{\psi_{\alpha,N}(r)}= \frac{1}{r^{2m-2}}\cdot\text{Volume of the zero set of
}\psi_{\alpha,N}\bigcap B(0,r), $$ where $B(0,r)=\{z\in \C^m:
|z|<r\}$. \par Our first main result will be the following:
\begin{theorem}
\label{MainCPN} Let $\psi_{\alpha,N}$ be a degree $N$ Gaussian
random SU($m+1$) polynomial,
$$\psi_{\alpha,N} (z)= \Sum_j \alpha_j
\sqrt{N \choose j} z^j,$$ where $\alpha_j$ are independent
identically distributed complex Gaussian random variables, and let
$n_{\psi_\alpha,N}(r) $ be the unintegrated counting function.\par
For all $\Delta
> 0,$ and $r>0$ there exists $A_{\Delta,r,m}$ and $N_{\Delta,r,m}$ such
that for all $N>N_{\Delta,r,m}$
$$\left|n_{\psi_\alpha,N}(r) - \frac{Nr^{2}}{1+r^2}
\right| \leq \Delta N$$ except for an event whose probability is
less than $e^{-A_{\Delta,r,m} N^{m+1}}. $

\end{theorem}
Theorem \ref{MainCPN} gives an upper bound on the rate of decay of
the hole probability, and we will be able to prove a lower bound
for the decay rate of the same order:
\begin{theorem}
\label{Hole probability CPN} Let $\psi_{\alpha,N}$ be as in
theorem \ref{MainCPN}, and let $$Hole_{N,r}=\{\alpha:\forall z \in
B(0,r), \ \psi_{\alpha,N} (z)\neq 0 \},$$ then there exists $
c_{1,r,m},$ $c_{2,r,m}>0$ and $N_r$ such that for all $N>N_{r,m}$
$$ e^{-c_{2,r} N^{m+1}} \leq Prob (Hole_{N,r})\leq e^{-c_{1,r} N^{m+1}}$$
\end{theorem}
As an immediate consequence of this result, the order of the
probability specified in the Theorem \ref{MainCPN} is the correct
order of decay.
\par
Random polynomials of the form studied here are the simplest
examples of a class of natural random holomorphic sections of
large $N$ powers of a positive line bundle on a compact K\"{a}hler
manifold. Most of the results stated in this paper may be restated
in terms of Szeg\"{o} kernels, which exhibit universal behavior in
the large $N$ limit in an appropriately scaled neighborhood.
Hopefully, this paper will provide insight into proving a similar
decay rate for this more general setting. This has already been
done for other properties of random holomorphic sections, e.g.
correlation functions, \cite{BleherShiffmanZelditchUniv}.\par

\par {\bf Acknowledgement:} I would like
to thank Bernard Shiffman and Misha Sodin for many useful
discussions.

\newsection{SU($m+1$) Invariance}\par

We begin by letting $Poly_N$ denote the set of polynomials in $m$
variables whose degree is less than or equal to $N$. $Poly_N$
becomes a
Hilbert space with respect to the following SU($m+1$) invariant norm, \cite{BleherShiffmanZelditchUniv}:\\

$$\displaystyle{\| f\|_{N}
^2}\displaystyle{:=}\displaystyle{\frac{N+m!}{N!
\pi^m}\int_{z\in\C^m} |f(z)|^2 \frac{dm(z)}{(1+|z|^2)^{N+m+1}}},$$
where $dm$ is just the usual Lebesque measure on $\C^m$. For this
norm $ \left\{ \sqrt{N \choose j} z^j\right\}$ is an orthonormal
basis, as is $ \left\{ \sqrt{N \choose j} \Prod_{k=0}^{m}
(\Sum_{l=1}^m a_{k,l} z_l+ a_{k,0})^{j_k} \right\}$, where
$A=(a_{k,l})$ and $A \cdot \overline{A^T}= I$, and $j_0 = N- |j|$.
Specifically, one alternate orthonormal basis is, for any
$\zeta\in \C^1 $, $$\left\{ \sqrt{N\choose j} \left(\frac{z_1 -
\zeta}{\sqrt{1+|\zeta|^2}}\right)^{j_1}
\left(\frac{1+\overline{\zeta}
z_1}{\sqrt{1+|\zeta|^2}}\right)^{N-|j|} \prod_{k=2}^m z_k^{j_k}
\right\}_{|j|\leq N}$$\par

Clearly, by line (1), a Gaussian random SU($m+1$) polynomial is
defined as, $\psi_{\alpha,N}(z)= \Sum_{|j|=0}^{|j|=N} \alpha_j
\psi_j(z)$, where $\alpha_j$ are i.i.d. standard complex Gaussian
random variables, and $\{\psi_j\}$ is the first orthonormal basis.
Any basis for $Poly_N$ could have been used and the Gaussian
random SU($m+1$) polynomials would be probabilistically identical,
as for $\{\alpha_j\}$ a sequence of i.i.d. Gaussian random
variables there exists another sequence of i.i.d. Gaussian random
variables, $\{\alpha_j'\}$, such that
$$\sum \alpha_j \sqrt{N \choose j} z^j =  \Sum_{|j|=0}^N \alpha'_j \sqrt{N \choose j} \left(\frac{z_1-\zeta}{\sqrt{1+|\zeta|^2}}\right)^{j_1}\left(\frac{1+ \overline{\zeta} z_1}{\sqrt{1+|\zeta|^2}}\right)^{N-|j|} \Prod_{k=2}^N z_{k}^{j_k} . \eqno{(2)} $$

\newsection{Large deviations of the maximum of a random SU($m+1$) polynomial}\par

In order to estimate $\displaystyle{\max_{B(0,r)}
\log|\psi_{\alpha,N}|}$, we will use following elementary
estimates to compute upper and lower bounds for the probability of
several events:

\begin{proposition}\label{Gauss}
Let $\alpha$ be a standard complex Gaussian random variable, \\
\begin{tabular}{cl}
then & i) $Prob(\{| \alpha| \geq \lambda \}) =
e^{-\lambda^2}$\\
& ii) $Prob(\{ | \alpha| \leq \lambda \}) =1- e^{-\lambda^2}\in
[\frac{\lambda^2}{2}, \lambda^2], if \lambda \leq 1$\\
& iii) if $\lambda\geq 1$ then $Prob(\{ | \alpha| \leq \lambda \})
\geq \haf$\\
\end{tabular}

\end{proposition}

This next lemma is key as it states that the maximum of the norm
of a random SU($m+1$) polynomial on the ball of radius $r$ tends
to not be too far from its expected value.
\begin{lemma}\label{GrowthRateCP1}\label{GrowthRateCP1NR0}
For all $\delta \in (0,1)$, and for all $r>0$ there exists
$a_{r,\delta,m}>0$ and $N_{\delta,m}$ such that for all $N>N_{\delta,m}$\\


 $$ \max_{B(0,r)} |\psi_{\alpha,N}(z)|
\in \left[(1+r^2)^{\frac{N}{2}}(1-\delta)^{\frac{N}{2}} ,
(1+r^2)^{\frac{N}{2}} (1+\delta)^{\frac{N}{2}}\right],$$ except
for an event whose probability is less than $e^{- a_{r,\delta,m}
N^{m+1}}$.

\end{lemma}

\begin{proof}
We will first prove a sharper decay estimate for the probability
of the event where a random SU($m+1$) polynomial takes on large
values in the ball of radius $r$:
$$Prob\left( \{ \max_{B(0,r)} |\psi_{\alpha,N}(z)|
>(1+r^2)^{\frac{N}{2}} (1+\delta)^{\frac{N}{2}} \}\right)<
e^{-c_m N^{2m}}.$$ To do this we consider the event $\Omega_N:= \{
\forall j, \ |\alpha_j| \leq N^m \},$ the complement of which has
probability $\leq (N+1)^m e^{-N^{2m}}$, by Proposition
\ref{Gauss}. For $\alpha \in \Omega_N$,\par
\begin{tabular}{rrl} $\displaystyle{\max_{z\in B(0,r)}|\psi_{\alpha,N}(z )|}$ &$\displaystyle{=}$&$\displaystyle{\max_{z\in B(0,r)} \left|\sum \alpha_j {N \choose j}^\haf (z)^j\right|}$ \\
&$\displaystyle{\leq}$ &$\displaystyle{\max_{z\in B(0,r)} \sum |\alpha_j| {N \choose j}^\haf |z|^j}$ \\
&$\displaystyle{\leq}$&$\displaystyle{ \max_{z\in B(0,r)} N^m
(N+1)^{\frac{m}{2}} (1+\sum |z_i|^2)^{\frac{N}{2}}}$,\\& & by the
Schwartz inequality.\\& $\displaystyle{ =}$& $\displaystyle{N^m
(N+1)^{\frac{m}{2}}(1+r^2)^{\frac{N}{2}}}$\\
& $\displaystyle{ \leq}$& $\displaystyle{
(1+\delta)^{\haf N}(1+r^2)^{\frac{N}{2}}}$, if $N>N'_{\delta,m,1 }$\\
\end{tabular}\par

In other words, if $N>N'_{\delta,m,1}$ then $$ \left\{
\max_{B(0,r)} |\psi_{\alpha,N}(z)|
>(1+r^2)^{\frac{N}{2}} (1+\delta)^{\frac{N}{2}} \right\} \subset \Omega_N^c$$ and thus, for all $N>N_{\delta,m}$, this first event has probability less than or equal to
 $(N+1)^m e^{-N^{2m}}< e^{-\haf N^{2m}}$. This decay rate is independent of $\delta$ and $r$, and the estimate for the order of the decay of this probability could be improved upon.
\par We complete the proof by showing that:

$$Prob\left( \{ \max_{B(0,r)} |\psi_{\alpha,N}(z)|
< (1+r^2)^{\frac{N}{2}}(1-\delta)^{\frac{N}{2}} \}\right)< e^{-
a_{r,\delta}N^{m+1} }.$$

This will be done when we prove the following claim concerning a
polydisk, $P(0,\frac{1}{\sqrt m}r ):= \{z\in\C^m: |z_1|<
\frac{1}{\sqrt m}r,\ |z_2|< \frac{1}{\sqrt m}r, \ldots,\ z_m <
\frac{1}{\sqrt m}r \}$:
$$Prob\left( \{ \max_{P(0,r)} |\psi_{\alpha,N}(z)| < (1+m
r^2)^{\frac{N}{2}}(1-\delta)^{\frac{N}{2}} \}\right)< e^{-
a_{r,\delta}N^{m+1}}.$$ This second claim is stronger as
$\displaystyle{\max_{P(0,\frac{1}{\sqrt{m}}r)}
|\psi_{\alpha,N}(z)|\leq \max_{B(0,r)} |\psi_{\alpha,N}(z)|}$.\par

Consider the event where $$\displaystyle{M=\max_{P(0,r)}
|\psi_{\alpha,N}(z)| < (1+mr^2)^{\frac{N}{2}}
(1-\delta)^{\frac{N}{2}}}.$$ We will show that this event can only
occur if certain Gaussian random variables, $\alpha_j$, obey the
inequality $|\alpha_j|<e^{-c_m N}$, where $c_m>0$. Further we will
show that this occurs whenever $j$ is in a certain cube which will
have sides of length $c_{\delta,r}N$. This will give us the
desired decay rate for the probability.

The Cauchy estimates for a holomorphic function state that:
$$\displaystyle{ |\psi_{\alpha,N}^{(j)}(0)|\leq j!
\frac{M}{r^{|j|}}}.$$ By differentiating equation (1) we compute
that $$ \psi_{\alpha,N}^{(j)}(0)= \sqrt{ N \choose{j}} j!
\alpha_j.$$ Combining this with Stirling's formula:
$$\sqrt{2\pi j_1}j_1^{j_1} e^{-j_1} <j_1! < \sqrt{2\pi
j_1}j_1^{j_1} e^{-j_1} e^{\frac{1}{12}}$$ we get that:
\par
\begin{tabular}{rll}$\displaystyle{ |\alpha_j |}$&$ \displaystyle{
\leq}$&$\displaystyle{ \frac{(1+mr^2)^{\frac{N}{2}} (1-\delta
)^\frac{N}{2}}{r^{|j|} \sqrt{ N \choose{j}}}}$ \\ &
$\displaystyle{ \leq}$ & $\displaystyle{ e^{\frac{m}{12} +
\frac{m}{2}\log(2\pi)} \left(\frac{(1+m r^2)^{\frac{N}{2}}
(1-\delta)^{\frac{N}{2}} (N-|j|)^{ \haf (N-|j| +\haf)}  \prod
(j_k)^{\frac{j_k+\haf}{2}}}
{r^{|j|} N^{\frac{N + \haf}{2}}}\right)}$\\
&$\displaystyle{ \leq}$&$\displaystyle{ e^{\frac{m}{12} +
\frac{m}{2}\log(2\pi)} N^{\frac{m}{4}} \cdot \left(\frac{(1+ m
r^2)^{\frac{N}{2}} (1-\delta)^{\frac{N}{2}}(N-|j|)^{ \haf (N-|j|)}
\prod (j_k)^{\frac{j_k}{2}} }
{r^{|j|} N^{\frac{N}{2}}}\right)}$\\
\end{tabular}\\ For the time being we focus on the term in parenthesis in the previous line which we call $A$.
Writing $j$ as $j= (j_k)= (x_k N), \ x_k\in (0,1)$, we now
have:\par

$\displaystyle{ A\ = \ (1-\delta)^\frac{N}{2}\left( \frac{(1+ m
r^2)}{r^{2|x| }} (1-|x|)^{(1-|x|)} \Prod_{k=1}^m\left(x_k
\right)^{x_k}\right)^\frac{N}{2}} $\\ \par If for all $k\in \{ 1,
2, \ldots, m\}$, $x_k=\frac{r^2}{1+ m r^2}$ then $A=
(1-\delta)^\frac{N}{2}$, which
inspires the following claim:\\
Claim: Let $s_{r,m}= \frac{1}{2m(1+m r^2)} \min\left\{
r^2, 1 \right\}$.\\
If for each $k\in \{ 1, 2, \ldots, m\}, \
x_k\in\left[\frac{r^2}{1+m r^2}-s_{r,m} \delta, \frac{r^2}{1+m
r^2}\right] \subset (0,1),$ and thus $|x|<1 $ then

$$(1+ m r^2)\left(\frac{x}{r^{2}}\right)^x (1-|x|)^{(1-|x|)}<
(1-\delta)^{-\frac{1}{2}}.$$\par

Proof: We begin by setting $x_k = (1-\Delta_k) \frac{r^2}{1+m
r^2}$ and $\Delta=\sum \Delta_k $. Therefore
$$\Delta_k \in \left[ 0, \frac{1}{2m}\min\left\{1, \frac{1}{r^2}
\right\}\delta \right]\ \text{and} \ \Delta \in \left[
0,\frac{1}{2}\min\left\{1, \frac{1}{r^2} \right\}\delta \right]
.$$ Thus, $1-|x| = \frac{1+\Delta r^2}{ 1+ m r^2}$, and from this
we compute that:
\\
\begin{tabular}{rcl} $\displaystyle{\frac{(1+m r^2)}{r^{2|x|}}
(1-|x|)^{(1-x)} (x_k)^{x_k}}$&
$\displaystyle{=}$& $\displaystyle{ \frac{1+m r^2}{r^{2|x|}} \left( \frac{1+\Delta r^2}{1+ m r^2}\right)^{1-|x|}}$\\
& & $\displaystyle{ \ \cdot  \left(\frac{r^2}{1+m r^2}\right)^{|x|} \prod(1-\Delta_k)^{x_k} } $\\
& $\displaystyle{=}$&$\displaystyle{ \left(1+\Delta r^2 \right)^{1-|x|} \prod (1-\Delta_k )^{x_k} }$\\
& $\displaystyle{\leq}$&$\displaystyle{\left(1+ \Delta r^2 \right)}$\\
& $\displaystyle{\leq}$& $\displaystyle{ 1+\haf \delta}$\\ &
$\displaystyle{ \leq} $&$\displaystyle{
\left(\frac{1}{1-\delta}\right)^\haf}$.
\end{tabular}\\ Proving the claim. \par

Therefore if for each $j, \displaystyle{\ x_j\in
\left[\frac{r^2}{1+m r^2}-s_{r,m} \delta, \frac{r^2}{1+m
r^2}\right]}$, then $A_k < (1-\delta)^{\frac{N}{4}}$. This in turn
guarantees that $|\alpha_j|< e^{\frac{m}{12} +
\frac{m}{2}\log(2\pi)} N^{\frac{m}{4}} (1-\delta)^{\frac{N}{4}}$.
The probability this occurs for a single $\alpha_j$ is less than
or equal to $$\left( e^{\frac{m}{12} + \frac{m}{2}\log(2\pi)}
N^{\frac{m}{2}} (1-\delta)^{\frac{N}{4}}\right)^2.$$ Thus the
probability this occurs for all $\alpha_j$, $j_k\in
\left[(\frac{r^2}{1+m r^2}-s_{r,m} \delta) N, (\frac{r^2}{1+m r^2}
)N \right]$, is less than or equal to
$$\left(e^{\frac{m}{12} +
\frac{m}{2}\log(2\pi)} N^{\frac{m}{4}}
(1-\delta)^{\frac{N}{4}}\right)^{2 (\lfloor N s_{r,m} \delta
\rfloor)^m }.$$

Hence, there exists $a_{r,\delta,m}>0$ and $N''_{\delta,m} $ such
that for all $N>N''_{\delta,m}$,

$$\left(e^{\frac{m}{12} +
\frac{m}{2}\log(2\pi)} N^{\frac{m}{4}}
(1-\delta)^{\frac{N}{4}}\right)^{2 (\lfloor N  s_{r,m} \delta
\rfloor)^m } <e^{-a_{r,\delta,m}N^{m+1}}$$ The result follows
after
setting $N_{\delta,m} = \max\{N'_{\delta,m}, N''_{\delta,m}\}$\\
\end{proof}

A nice application of this lemma, along with line (2) is the
following:

\begin{lemma}\label{GrowthRateCP1General}\label{ValueEstimateCP1}
For all $\Delta \in (0,1)$ and $a \in \C^n\backslash\{0 \} $ there
exists $N_{\Delta,|a|,m}$ and $c_{\Delta,|a|,m}>0 $, such that if
$N>N_{\Delta, |a|,m}$ then
$$\displaystyle{ \max_{z\in B(0,
\Delta )} |\psi_{\alpha,N}(z-a)| < (1+ |a|^2)^{\frac{N}{2}} (1-
\Delta)^{\frac{N}{2}}},$$ except for an event whose probability is
less than $e^{-c_{\Delta, |a|,m} N^{m+1}}$.\par

\end{lemma}
\begin{proof}As Gaussian random SU($m+1$) polynomials are rotationally
invariant, as a random process, with out loss of generality we
assume that a is of the form: $a= (\zeta_1, 0, \ldots, 0)$.
\par Let
$\delta=\frac{\Delta}{2+ 2 |\zeta|+2|\zeta|^2)}$.\\

By Lemma \ref{GrowthRateCP1NR0} and line (2), there exists
$c_{\Delta,m}>0 \ \text{and} \ N_{\Delta,|a|, m}$ such that if $N>
N_{\Delta,|a|, m}$ then, except for an event whose probability is
less than $e^{- c_{\Delta,|a|} N^{m+1}}$,
\par \begin{tabular}{rcl}
$\displaystyle{(1-\delta)^{\frac{N}{2}}}$& $ \leq$ &
$\displaystyle{\frac{\max_{B(0,\delta)}
|\psi_{\alpha,N}(z)|}{(1+\delta)^\frac{N}{2}} }$\\& $=$& $
\displaystyle{\max_{\partial B(0,\delta)}} \frac{\left| \sum
\alpha'_j \sqrt{N \choose j}\left(\frac{z_1-\zeta_1}{\sqrt{1+
|\zeta_1|^2}}\right)^{j_1}
\left(\frac{1+\overline{\zeta_1}z_1}{\sqrt{1+|\zeta|^2}}\right)^{N-|j|} \prod (z_k)^{j_k}\right|}{(1+\delta^2)^{\frac{N}{2}}}.$\\
\end{tabular}
 In
order to simplify this previous line, let
$$\phi(z)=\left(\frac{z_1-\zeta_1}{1+\overline{\zeta_1}z_1}, \frac{z_2
(\sqrt{1+|\zeta|^2})}{1+\overline{\zeta_1}z_1}, \ldots, \frac{z_m
(\sqrt{1+|\zeta|^2})}{1+\overline{\zeta_1}z_1}\right),$$ so that
we may rewrite the previous equation as:\par

\begin{tabular}{rrl}
$\displaystyle{(1-\delta)^{\frac{N}{2}} }$&$\leq$&$\displaystyle{\left(\max_{\partial B(0,\delta)} \frac{|1+\overline{\zeta_1} z_1|^N}{(1+ |\zeta_1|^2)^{\frac{N}{2}}(1+\delta^2)^{\frac{N}{2}}} \right) \left(\max_{ B(0,\delta)}\left| \psi_{\alpha',N}(\phi(z))\right|\right) }$\\

&$\leq$&$\displaystyle{\left( \frac{(1+|\zeta_1| \delta)^N}{((1+
|\zeta_1|)^2(1+\delta^2))^{\frac{N}{2}}} \right)
\left(\max_{B(-\zeta_1,(4+2
|\zeta_1|^2)\delta)} \left| \psi_{\alpha',N}(z)\right|\right)}$,\\
\end{tabular}\\
as the image of $\phi\mid_{B(0,\delta)}\subset B(-\zeta_1, (2+2
|\zeta_1|^2)\delta)$, since:\\
\begin{tabular}{rrl}
$\displaystyle{\max_{z\in \partial B(0,\delta)} \left|
\frac{z_1-\zeta}{1+ \overline{\zeta}z_1}+
\zeta\right|^2}$&$\displaystyle{ +}$&$\displaystyle{\Sum_{k>0}
\left| \frac{z_k
(\sqrt{1+|\zeta|^2})}{1+ \overline{\zeta}z_1} \right|^2}$\\
& $\displaystyle{=}$&$\displaystyle{ \max_{z\in
\partial B(0,\delta)} \left| \frac{z_1-\zeta +\zeta +
z_1|\zeta|^2}{1+ \overline{\zeta}z} \right|^2+ \Sum_{k>0} |z_k|^2
\left|
\frac{\sqrt{1+|\zeta|^2}}{1+ \overline{\zeta}z_1} \right|^2} $\\
& $\displaystyle{=}$&$\displaystyle{ \delta^2 \max_{z\in \partial B(0,\delta)} \left| \frac{ (-1 -\zeta^2)}{1+ \overline{\zeta}z_1} \right|^2 }$\\
& $\displaystyle{\leq}$&$\displaystyle{ 4\delta^2  (|\zeta|^2 +1
)^2}$\\
\end{tabular}\\

Rearranging the previous sets of equations we get the result:\\
\begin{tabular}{rrl}
$\displaystyle{\max_{ B(0,\Delta)} |\psi_{\alpha',N}(z-\zeta_1)| }
$

& $\displaystyle{\geq}$&$\displaystyle{
\frac{(1+|\zeta_1|^2)^{\frac{N}{2}}(1+\delta^2)^{\frac{N}{2}} }{
{(1+|\zeta_1| \delta)^N}  }
\cdot (1-\delta)^{\frac{N}{2}} } $ \\

& $\displaystyle{\geq}$&$\displaystyle{
(1+|\zeta_1|^2)^{\frac{N}{2}}  (1- (2+2|\zeta_1|)
\delta)^{\frac{N}{2}}
 } $ \\

&$\displaystyle{\geq}$&$\displaystyle{
(1-\Delta)^{\frac{N}{2}}(1+|\zeta_1|^2)^{\frac{N}{2}}} $
\end{tabular}\\
\end{proof}

\newsection{Second key lemma}\par
The goal of this section will be to estimate $\int_{S_r} \log
|\psi_{\alpha,N}(z)| d\mu_r(z) $, where $d\mu_r(z)$ is the
rotationally invariant probability measure of the sphere of radius
$r$, $S_r= \partial B(0,r)$, which will be accomplished when we
prove lemma \ref{SurfaceIntegralEstimateOnS1}, using the same
techniques as in \cite{SodinTsirelson05}. As $\log(x)$ becomes
unbounded near 0, we will first prove a deviation result for the
event where the $L^1$ norm of $\log|\psi_{\alpha,N}|$ is
significantly larger than its max on the same region.
\begin{lemma} \label{oakleyCP1}For all $r>0$ there exists $c_{m,r}, \ \text{and} \ N_m$ such that for all $N>N_m$,
$$\int_{S_r } |\log(|\psi_{\alpha,N} (z)|)| d \mu_r (z) \leq
  \left(\frac{3^{2m}}{2}+\haf\right) N  \log\left((2) (1+r^2)\right)$$ except for an event whose probability is $<e^{-c_{m,r}
  N^{m+1}}$.
\end{lemma}

\begin{proof}
By Lemma \ref{GrowthRateCP1}, there exists $c_{m,r}, \ \text{and}
\ N_{m,r}$ such that if $N>N_{m,r}$ then, with the exception of an
event whose probability is less than $e^{-c_{m,r} N^{m+1}}$, there
exists $\zeta_0 \in
\partial B(0,\frac{1}{2}r)$ such that $\log(|\psi_{\alpha,N} (\zeta_0)|) >
0$. This also implies that:
$$\int_{z\in S_r} P_r(\zeta_0, z) \log(|\psi_{\alpha,N}(z)|) d\mu_r (z)\geq \log(|\psi(\zeta_0)|)\geq
0,$$Where $P_r$ is the Poisson kernel for the sphere of radius r:
$P_r(\zeta, z)= r^{2m-2} \frac{r^2-|\zeta^2|}{|z-\zeta|^{2m}}$.
Hence,
$$\int_{z\in S_r} P_r(\zeta_0, z) \log^-(|\psi_{\alpha,N}(z)|) d\mu_r(z) \leq \int_{z\in S_r}
P(\zeta_0, z) \log^+(|\psi_{\alpha,N}(z)|) d\mu_r (z)$$
\par
Now given the event where
$$\displaystyle{\log\max_{B(0,r)}}|\psi_\alpha(z)|< \frac{N}{2}
\log \left((2) (1+r^2)\right),$$ (whose complement for $N>N_{m,r}$
has probability less than $e^{-c_m N^{m+1}}$), we may estimate
that
$$\displaystyle{ \int_{z\in S_r} \log^+(|\psi_{\alpha,N}(z)|) d \mu_r (z) \leq \frac{N}{2} \log \left((2) (1+r^2)\right)}.$$ Since
$\zeta_0 \in
\partial B(0, \haf r)$ and $z = r e^{i\theta}$, we have:  $
\frac{r^2}{4} \leq |z- \zeta_0 |^2 \leq \frac{9}{4} r^2 $. Hence,
by using the formula for the Poisson Kernel,
$$\displaystyle{\frac{2^{2m-2}}{3^{2m-1}} \leq P(\zeta_0, z)\leq 3
\cdot 2^{2m-2} }.$$ Putting the pieces together
proves the result:\\
$$\displaystyle{\int_{z\in S_r} P_r(\zeta_0,
z)\log^+}(|\psi_{\alpha,N}(z)|)d \mu_r \leq 3 \cdot 2^{2m-3} N
\log\left( 2(1+r^2)\right) $$
\begin{tabular}{rrl}
$\displaystyle{\int_{z\in S_r} \log^-(|\psi_{\alpha,N}(z)|) d \mu_r (z)}$&$\displaystyle{\leq}$& $\displaystyle{ \frac{1}{\min P(\zeta_0, z)} \int_{z\in S_r} P_r(\zeta_0, z)\log^+(|\psi_{\alpha,N}|) d\mu_r(z) }$\\
& $\displaystyle{\leq}$& $\displaystyle{
 \frac{3^{2m}}{2} N \log\left(2 (1+r^2)\right)} $ \\

\end{tabular}
 \end{proof}

We now arrive at the main result of this section:
\begin{lemma}\label{SurfaceIntegralEstimateOnS1} For all $r>0$ and for all $\Delta\in (0, 1)$ there exists $c_{\Delta,r,m}>0$ and $N_{\Delta,r,m} $ such
that for all $N>N_{\Delta,r,m}$,
$$\int_{z\in S_r}
\log(|\psi_{\alpha,N}(z)|) d \mu_r (z) > \frac{N}{2} \log
\left((1+r^2) (1-\Delta) \right),$$ except for an event whose
probability is less than $e^{-c_{\Delta,r,m} N^{m+1}} $.
\end{lemma}
\begin{proof}
It suffices to prove this result for small $\Delta$. Set
$\displaystyle{\delta= 3^{-4m}\Delta^{4m}}$. Let $s= \lceil
\frac{1}{\delta} \rceil$, let $Q=(2m) s^{2m-1}$, and let
$\displaystyle{\kappa=1-\delta^{\frac{1}{4m}}}$.\par In
\cite{Zrebiec06} it was shown that by projecting a tiling of the
$2m$ cube by $2m-1 $ cubes onto the sphere of radius $\kappa r$
one gets a partition consisting of $Q$ measurable disjoint sets
$\{I^{\kappa r}_1, I^{\kappa r}_2, \ldots, I^{\kappa r}_Q\}$ such
that
$$diam(I^{\kappa r_j})\leq \frac{\sqrt{2m-1}}{s}\kappa r=
  \frac{c_m}{Q^{\frac{1}{2m-1}}}\kappa r.$$ We choose such a partition and then we choose a $\zeta_j$ within $\delta r<1 $ of
 $I_j^{\kappa r}$ such that
$$\log (|\psi_{\alpha,N}(\zeta_j)|) >  \frac{N}{2} \log \left((1 +  \kappa^2 r^2) (1-\delta r)\right), \eqno{(3)}$$
for which, by Lemma \ref{GrowthRateCP1General}, there exists
$c'_{\Delta,r}$ and $N'_{\Delta,r}$ such that if $N>N_{\Delta,r}$
then the probability that this  does not occur is less than $
e^{-c'_{\delta,r} N^{m+1}}$. Therefore there exists
$c_{\Delta,r}>0$ and $N_{\Delta,r}$ such that if $N>N_\Delta $ the
union of these m events has probability less than or equal to
$$\left(2m \left\lceil \frac{1}{\delta} \right\rceil^{2m-1}\right)
e^{-c'_{\delta,r} N^{m+1}} < e^{-{c_\delta,r}N^{m+1}}.
\eqno{(4)}$$ Let $\mu_k ={\mu_{\kappa r} (I_k^{\kappa r})} $. As
$\{I_1^{\kappa r}, I_2^{\kappa r}, \ldots, I_Q^{\kappa r}\}$ form
a partition of $S_{\kappa r}$, $\Sum_k \mu_k= 1$.\par We now turn
to investigating the average of $\log |\psi_{\alpha,N}(z)|$
on the sphere of radius r by approximating said integral with a Riemann sum which makes use of line (3):\\
\\
\begin{tabular}{cl}
$\displaystyle{ \frac{N}{2} \log( (1 + \haf \kappa^2 r^2)}$ &
$\displaystyle{ (1-\delta))\leq \Sum_{k=1}^{k=Q} \mu_k
\log|\psi_{\alpha,N}(\zeta_k)| }$\\& $ \displaystyle{\leq \
\int_{z\in S_r} \left( \Sum_k \mu_k P_r(\zeta_k, z)
\log(|\psi_{\alpha,N} (z)|) d \mu_r (z)\right) }$
\\
& $\displaystyle{= \ \int_{z\in S_r} \left(\Sum_k \mu_k
(P_r(\zeta_k, z) -1) \right) \log(|\psi_{\alpha,N}(z)|) d \mu_r(z) }$\\
&$\displaystyle{ \ \ \ \  + \int_{z\in S_r} \log(|\psi_{\alpha,N}
(z)|) d \mu_r(z)}$
\\
\end{tabular}\\
This will simplify to:\\
\begin{tabular}{rl}$\int \log(|\psi_{\alpha,N}(z)|)
d \mu_r (z)\geq$&$\frac{N}{2}\log \left((1 + \kappa^2 r^2 ) (1-
\delta r)\right)$\\ &$ -  ( \int | \log|\psi_{\alpha,N}(z)|| d
\mu_r (z)) \displaystyle{\max_{z\in S_r}} |\sum_k \mu_k
(P_r(\zeta_k, z)-1) |$
\end{tabular}
\\

In \cite{Zrebiec06}, it was computed that in exactly this situation that:\\
$$\displaystyle{ \max_{ |z|=r} \left|  \sum_k \mu_k
(P_r(\zeta_k, z) -1 )\right| \leq C_m \delta^{\frac{1}{2(2m-1)
}}}$$

Hence by Lemma \ref{oakleyCP1} and line (4), there exists $c_{\delta, r, m}>0\ \text{and} \ N_{\delta,r,m}$ such that if $N>N_{\delta,r,m}$, except for an event of probability $<e^{-c_{\delta,r,m}N^{m+1}}$:\\
\begin{tabular}{rcl}$\displaystyle{\int \log(|\psi_{\alpha,N}|) d\mu_r (z)
}$ & $\displaystyle{\geq}$ & $\displaystyle{
\frac{N}{2}\log\left((1 + \kappa^2 r^2) (1- \delta r)\right)}$\\&&
$\ \ \displaystyle{ - C_m N
\log \left( 2(1+r^2) \right)  \delta^{\frac{1}{2(2m-1)}},}$\\

&$\displaystyle{=}$&$\displaystyle{ \frac{N}{2}\log\left((1+r^2)-2
\delta^{\frac{1}{4m}} r^2 +
O(r^2\delta^{\frac{1}{2m}}+\delta^{\frac{4m+1}{4m}} r^3+ \delta r)
\right)}$\\
&&$\ \ \displaystyle{ - C_m N
\log \left( 2(1+r^2) \right)  \delta^{\frac{1}{2(2m-1)}},}$\\

&$\displaystyle{ \geq}$&$\displaystyle{ \frac{N}{2}\log((1+r^2)
(1-3\delta^\frac{1}{4m}))},$ for sufficiently small $\delta$.
\end{tabular}\\
The proof is thus completed by choosing sufficiently small
$\Delta$ so that the previous line holds, (and $\delta r<1 $) .\\
\end{proof}

\newsection{Main Results}\par
We will now be able to estimate the value of the unitegrated
counting function for a random SU($m+1$) polynomial,
$\psi_{\alpha, N} $.

\begin{deff} For $f\in \holo (B(0,r)), \ f(0)\neq 0, \ B(0,r) \subset \C^m$, the
unintegrated counting function,\\
$n_{f}(r):= \int_{B(0,t)\bigcap Z_{f}} (\frac{i}{2 \pi}
\partial \overline{\partial} \log |z|^2)^{m-1}= \int_{B(0,t)}
(\frac{i}{2 \pi} \partial \overline{\partial} \log |z|^2)^{m-1}
\wedge \frac{i}{2 \pi} \partial \overline{\partial} \log |f| $
\end{deff}
The equivalence of these two definitions follows by the
Poincare-Lelong formula. The above form ($(\frac{i}{2 \pi}
\partial \overline{\partial} \log |z|^2)^{m-1} $) gives a
projective volume, with which it is more convenient to measure the
zero set of a random function. The Euclidean volume may be
recovered as $$\int_{B(0,t)\bigcap Z_{f}} \left(\frac{i}{2 \pi}
\partial \overline{\partial} \log |z|^2\right)^{m-1}=
\int_{B(0,t)\bigcap Z_{f}} \left(\frac{i}{2 \pi t^{2}} \partial
\overline{\partial} |z|^2\right)^{m-1}.$$

\begin{lemma}\label{Shabat}If $u \in L^1(\overline B _r), \
and \ \partial \overline{\partial}u$ is a measure, then
$$\int_{t=r\neq 0}^{t=R} \frac{dt}{t} \int_{B_t} \frac{i}{2 \pi} \partial
\overline{\partial} u \wedge (\frac{i}{2 \pi} \partial
\overline{\partial} \log |z|^2)^{m-1}= \frac{1}{2} \int_{S_R}u
d\mu_R - \frac{1}{2} \int_{S_r}u d\mu_r
$$
\end{lemma}
\par A proof of this result is available on page 390-391 of Griffiths
and Harris, \cite{GriffithsAndHarris}. Using this we
may now prove one of our two main theorems, Theorem \ref{MainCPN}: \\

\begin{proof} (of theorem \ref{MainCPN}). It suffices to prove the result for small $\Delta$. Let $\delta= \frac{\Delta^2}{4}<1 $. Let $\kappa = 1 +
 \sqrt{\delta}= 1+ \frac{\Delta}{2}
 $. As $n_{\psi_{\alpha,N}}(r)$ is increasing,
$$n_{\psi_{\alpha,N}}(r) \log(\kappa) \leq \int_{t=r}^{t=\kappa r}
n_{\psi_{\alpha,N}}(t)\frac{dt}{t} \leq n_{\psi_{\alpha,N}}(\kappa
r) \log(\kappa) \eqno{(3)}$$ \par

There exists $c_{\delta, r,m}>0 \ \text{and} \ N_{\delta, r,m}$ such that for all $N>N_{\delta,r,m}$, except for an event of probability $\displaystyle{\leq e^{-c_{\delta,r,m} N^{m+1}}}$, we get that:\\
\begin{tabular}{rl}
$\displaystyle{ n_{\psi_{\alpha,N}}(r) \log(\kappa)}$ &$\displaystyle{\leq \int_{S_{\kappa r}} \log| \psi_{\alpha,N} (z)| d\mu_{\kappa r}(z) - \int_{S_r} \log| \psi_{\alpha,N} (z)| d\mu_r(z)}$\\
& $\displaystyle{ \leq \frac{N}{2} \left( \log \left( (1+ \kappa^2
r^2) (1+\delta)\right) - \int_{S_r} \log |\psi_\alpha ( r
e^{i\theta})| d\mu_r \right) }$,\\& by Lemma \ref{GrowthRateCP1}.
\\ & $\displaystyle{ \leq
\frac{N}{2} \left( \log \left( (1+ \kappa^2 r^2 ) (1+\delta)
\right) - \log \left( (1+r^2)
(1-\delta)\right)\right)}$,\\&   by Lemma \ref{SurfaceIntegralEstimateOnS1}.\\
& $\displaystyle{ \leq \frac{N}{2}   \left( \frac{2\sqrt{\delta}
r^2+ \delta r^2 + 2 \delta +2\delta r^2  }{(1+r^2) } - \frac{2
\delta
r^4}{(1+r^2)^2} + O (\delta^\frac{3}{2}) \right) }$, \\
\end{tabular}\par
Therefore,\\
\begin{tabular}{rcl}$\displaystyle{n_{\psi_{\alpha,N}}(r)}$&$\leq$&$\displaystyle{  N  \left( \frac{
r^2+ \haf \sqrt\delta r^2 + \sqrt\delta + \sqrt{\delta}r^2
}{(1+r^2) } - \frac{ \sqrt \delta r^4}{(1+r^2)^2} + O (\delta)
\right)}$\\&& $\displaystyle{\ \ \ \ \ \ \ \ \ \ \ \ \ \ \ \ \ \ \
\ \cdot\left( 1 + \frac{\sqrt\delta}{2} +
O(\delta)\right) }$\\
&$\leq$&$\displaystyle{   \frac{N r^2}{1+r^2}+ 3 N \sqrt
\delta + O (\delta) }$\\
\end{tabular}

This proves the probability estimate when the value of the
unintegrated counting function $n_{\psi_{\alpha,N}}(r) $ is
significantly above its typical value. We now modify the above the argument to finish the proof. There exists $c_{\delta,r,m} \ \text{and} \ N_{\delta,r,m}$ such that if $N> N_{\delta,r,m}$ then, except for an event whose probability is less than $e^{-c_{\delta,r,m}N^{m+1}}$, the following inequalities hold: \\
\begin{tabular}{rl}
$\displaystyle{ n_{\psi_{\alpha,N}}(r)}$ &$\displaystyle{ { \log(\kappa)} \geq \int_{S_{r}} \log| \psi_{\alpha,N} (z)| d\mu_r(z) - \int_{S_\kappa^{-1}r} \log| \psi_{\alpha,N} (z)| d\mu_{\kappa^{-1} r}(z)}$\\
& $\displaystyle{ \geq \frac{N}{2} \left( \log \left( (1+ r^2)
(1-\delta)\right) - \int_{S_r} \log |\psi_\alpha ( r e^{i\theta})|
d\mu_r \right) }$, by Lemma
\ref{SurfaceIntegralEstimateOnS1}.\\
& $\displaystyle{ \geq \frac{N}{2} \left( \log \left( (1+ r^2 )
(1-\delta) \right) - \log \left( (1+\kappa^{-2}r^2)
(1+\delta)\right)\right)}$,\\& by Lemma \ref{GrowthRateCP1}.
\\
& $\displaystyle{ \geq \frac{N}{2} \left( \log  (1-\delta) -
 \log \left( 1- \frac{2 \sqrt{\delta} r^2}{1+r^2}+ \frac{\delta r^2}{1+r^2} +\delta+ O(\delta^{\frac{3}{2}}) \right)\right)}$.\\
\end{tabular}\par

Therefore,\\
\begin{tabular}{rcl}$\displaystyle{n_{\psi_{\alpha,N}}(r)}$&$\geq$&$\displaystyle{N \left( -\sqrt\delta + \frac{2 r^2}{1+r^2} - \frac{\sqrt \delta r^2}{1+r^2} - \sqrt\delta + \frac{2 \sqrt\delta r^4}{(1+r^2)^2} + O(\delta) \right)}$\\&& $\displaystyle{\ \ \ \ \ \ \ \ \ \ \ \ \ \ \ \ \ \ \
\ \cdot\left( 1 + \frac{\sqrt\delta}{2} +
O(\delta)\right) }$\\
&$\geq$&$\displaystyle{   \frac{N r^2}{1+r^2}- 2 N \sqrt
\delta + O (\delta) }.$\\
\end{tabular}

\end{proof}

We have just implicitly proven an upper bound on the order of the
decay of the hole probability. We will now compute the lower bound
to finish the proof theorem \ref{Hole probability CPN}

\begin{proof} (of theorem \ref{Hole probability CPN}) The desired
upper bound for the order of the decay of the hole probability is
a consequence of the previous theorem.
\par We must still prove the lower bound for the order of the decay of the hole probability, and we start this by
considering the event, $\Omega$ which consists of $\alpha_j$
where:\par\begin{tabular}{cl} & $\displaystyle{|\alpha_0|\geq 1}$\\
&$\displaystyle{|\alpha_j|< {N \choose j}^\frac{-1}{2}
N^{-m}r^{-|j|}}$.
\end{tabular}\\
If $\alpha\in \Omega$, then $\displaystyle{|\alpha_0| >
\sum_{|j|>0}^N |\alpha_j| {N \choose j}^\frac{1}{2} r^j}$. Hence
for all $ z \in B(0,r),\  \psi_{\alpha,N}(z)\neq 0
 \Rightarrow \Omega \subset Hole_{N,r} $. A lower bound for the probability of $\Omega$ will thus give a
 lower bound for the probability of $Hole_{N,r}$. First we restrict
ourselves to considering the Gaussian random variables, $\alpha_j$, for whose indices,
$j$, ${N \choose j} N^{-2m} r^{-2|j|}\leq 1$.\\
\begin{tabular}{rrl}$\displaystyle{Prob\left(\left\{ \ |\alpha_j|<{N \choose
j}^\frac{-1}{2} N^{-m} r^{-|j|} \right\}
\right)}$&$\displaystyle{\geq}$&$\displaystyle{
\frac{1}{2} \frac{1}{N^{2m}{N \choose j} r^{2|j|}}}$,\\
&&by Proposition \ref{Gauss}.\\
&$\displaystyle{=}$&$\displaystyle{ \frac{1}{2} \frac{(N-|j|)!
j!}{N^{2m} N!} r^{-2|j|}}$\\
&$\displaystyle{\geq}$&$\displaystyle{\frac{
(2\pi)^{\frac{m-1}{2}} }{ N^{2m} 2 (m+1)^{N+\frac{m}{2}} }
 r^{-2|j|} e^{\frac{1}{12}}}$
\\ &$\displaystyle{\geq}$& $\displaystyle{ e^{-(N+\frac{m}{2})\log(m+1)+c_m -
|j|\log(r) - 2m \log (N)}}$
\\ &$\displaystyle{\geq}$& $\displaystyle{ e^{-(N)\log((m+1)r^{\frac{|j|}{N}})+c'_m - 2m \log (N)}}$
\\ &$\displaystyle{\geq}$& $\displaystyle{ e^{-c_{m,r}(N)}}$
\end{tabular}
\par
Please note that the last inequality still holds even if
$r<\frac{1}{m+1}$ since
$$\displaystyle{Prob\left(\left\{ \ |\alpha_j|<{N \choose
j}^\frac{-1}{2} N^{-m} r^{-|j|} \right\} \right)}\approx
e^{c_{m,r}N}\geq e^{-c'_{m,r}N}$$

Whereas if for the index $j$, ${N \choose j} N^{-2m}r^{-2|j|}<1$
then
\par
\begin{tabular}{rcl}$\displaystyle{Prob\left(\left\{ \ |\alpha_j|<{N \choose
j}^\frac{-1}{2} N^{-m} r^{-|j|} \right\}
\right)}$&$\displaystyle{\geq}$&$\displaystyle{ Prob\left(\left\{
\ |\alpha_j|<1 \right\} \right)}$ \\&$\displaystyle{>}$&$
\displaystyle{\haf}$\\ &$\displaystyle{>}$ &$\displaystyle{
e^{-N\log(2)}} $
\end{tabular}

Further, $Prob(\{|\alpha_0|
> N \})= e^{-1}$. Hence, $\displaystyle{Prob (\Omega)\ \geq} \ e^{-c_{r,m} N^{m+1} }$\\

\end{proof}

\end{document}